\documentclass[10pt,reqno]{amsart}
\usepackage[hypertex]{hyperref}

\usepackage{amsmath,amsthm}
\usepackage{amssymb}
\usepackage{euscript}

\numberwithin{equation}{subsection}

\newcommand{\sqsp}{\renewcommand{\baselinestretch}{1.15}\tiny\normalsize}

\newcommand{\relphantom[1]}{\mathrel{\phantom{#1}}}

\newtheorem{theorem}[subsection]{Theorem}
\newtheorem{lemma}[subsection]{Lemma}
\newtheorem{proposition}[subsection]{Proposition}
\newtheorem{corollary}[subsection]{Corollary}

\theoremstyle{definition}
\newtheorem{definition}[subsection]{Definition}
\newtheorem{example}[subsection]{Example}

\newcommand{\bk}{\mathbf{k}}

\newcommand{\muop}{\mu^{op}}
\newcommand{\mualpha}{\mu_\alpha}
\newcommand{\mubeta}{\mu_\beta}
\newcommand{\mun}{\mu^{(n)}}

\newcommand{\defn}{\buildrel \text{def} \over =}

\DeclareMathOperator{\Hom}{Hom}

\begin{document}

\title{Right Hom-alternative algebras}
\author{Donald Yau}

\begin{abstract}
It is shown that every multiplicative right Hom-alternative algebra is both Hom-power associative and Hom-Jordan admissible.  Multiplicative right Hom-alternative algebras admit Albert-type decompositions with respect to idempotents.  Multiplication operators defined by idempotents in right Hom-alternative algebras are studied.  Hom-versions of some well-known identities in right alternative algebras are proved.
\end{abstract}



\address{Department of Mathematics\\
    The Ohio State University at Newark\\
    1179 University Drive\\
    Newark, OH 43055, USA}
\email{dyau@math.ohio-state.edu}

\date{\today}
\maketitle

\sqsp

\section{Introduction}

An algebra that satisfies
\[
(xy)y = x(yy)
\]
is called a right alternative algebra.  If a right alternative algebra also satisfies the left alternative identity
\[
(xx)y = x(xy),
\]
then it is called an alternative algebra.  For example, the $8$-dimensional Cayley algebras are alternative algebras that are not associative  \cite{schafer}.  Alternative algebras are closely related to other classes of non-associative algebras.  In fact, alternative algebras are Jordan-admissible, Maltsev-admissible \cite{bk,maltsev}, and power associative \cite{albert1,albert2}.  Alternative algebras also satisfy the Moufang identities \cite{schafer}.

Generalizations of (left/right) alternative algebras, called (left/right) Hom-alternative algebras, were introduced by Makhlouf \cite{mak}.  These (left/right) Hom-alternative algebras are defined by relaxing the defining identities in (left/right) alternative algebras by a linear self-map, called the twisting map.  Construction results and examples of Hom-alternative algebras can be found in \cite{mak,yau4}.  Hom-type generalizations of classical algebras appeared in \cite{hls} in the form of Hom-Lie algebras, which were used to describe deformations of the Witt algebra and the Virasoro algebra.  Hom-associative algebras were defined in \cite{ms} and further studied in \cite{ms2,yau0,yau1}.  For other Hom-type algebras, see \cite{ms,ms2,yau2,yau3,yau4,yau15} and the references therein.

Many properties of alternative algebras have Hom-type generalizations.  Indeed, the author proved in \cite{yau4} that multiplicative Hom-alternative algebras are Hom-Jordan admissible and Hom-Maltsev admissible, and they satisfy Hom-versions of the Moufang identities.  Moreover, in \cite{yau15} the author showed that, as an immediate consequence of a much more general result, multiplicative right Hom-alternative algebras are Hom-power associative.

The purpose of this paper is to study right Hom-alternative algebras.  In section \ref{sec:ex} it is observed that the category of right Hom-alternative algebras is closed under twisting by self-weak morphisms (Theorem \ref{thm:twist}).  A construction result in \cite{mak} of multiplicative right Hom-alternative algebras is then recovered as a special case (Corollary \ref{cor2:twist}). In Example \ref{ex:5d} an infinite family of distinct isomorphism classes of multiplicative right Hom-alternative algebras are constructed.  Each of these right Hom-alternative algebras is neither left Hom-alternative nor right alternative.

In section \ref{sec:hompower} it is shown with a short proof that every multiplicative right Hom-alternative algebra is Hom-power associative (Theorem \ref{thm:power}).  In section \ref{sec:homjordan} it is shown that every multiplicative right Hom-alternative algebra is Hom-Jordan admissible (Theorem \ref{thm:jordan}).  Notice that in this result, left Hom-alternativity is not needed.

In section \ref{sec:decomp} the Hom-version of an idempotent is defined.  A generalization of Albert's decomposition is proved for multiplicative right Hom-alternative algebras  (Proposition \ref{prop:decomp}).  In section \ref{sec:idempotent} various multiplication operators on right Hom-alternative algebras induced by idempotents are studied.

In section \ref{sec:identity} the Hom-versions of some well-known identities, including a Moufang identity, in right alternative algebras are proved.

\section{Examples of right Hom-alternative algebras}
\label{sec:ex}

The purposes of this section are to prove some construction results for right Hom-alternative algebras and to provide some examples of right Hom-alternative algebras that are neither left Hom-alternative nor right alternative.

\subsection{Notations}

Throughout the rest of this paper, we work over a fixed field $\bk$ of characteristic $0$.  If $V$ is a $\bk$-module and $\mu \colon V^{\otimes 2} \to V$ is a bilinear map, then $\muop \colon V^{\otimes 2} \to V$ denotes the opposite map, i.e., $\muop = \mu\tau$, where $\tau \colon V^{\otimes 2} \to V^{\otimes 2}$ interchanges the two variables.  For $x$ and $y$ in $V$, we sometimes write $\mu(x,y)$ as $xy$.  For a linear self-map $\alpha \colon V \to V$, denote by $\alpha^n$ the $n$-fold composition of $n$ copies of $\alpha$, with $\alpha^0 \equiv Id$.

We now provide some basic definitions regarding Hom-algebras.

\begin{definition}
\label{def:homalg}
\begin{enumerate}
\item
A \textbf{Hom-module} is a pair $(A,\alpha)$ in which $A$ is a $\bk$-module and $\alpha \colon A \to A$ is a linear map, called the twisting map.  A \textbf{morphism} $f \colon (A,\alpha_A) \to (B,\alpha_B)$ of Hom-modules is a linear map of the underlying $\bk$-modules such that
\[
f\alpha_A = \alpha_Bf.
\]
\item
A \textbf{Hom-algebra} is a triple $(A,\mu,\alpha)$ in which $(A,\alpha)$ is a Hom-module and $\mu \colon A^{\otimes 2} \to A$ is a bilinear map, called the multiplication.  A Hom-algebra $(A,\mu,\alpha)$ and the corresponding Hom-module $(A,\alpha)$ are often abbreviated to $A$.
\item
A Hom-algebra $(A,\mu,\alpha)$ is \textbf{multiplicative} if $\alpha$ is multiplicative with respect to $\mu$, i.e.,
\[
\alpha\mu = \mu \alpha^{\otimes 2}.
\]
\item
Let $A$ and $B$ be Hom-algebras.  A \textbf{weak morphism} $f \colon A \to B$ of Hom-algebras is a linear map $f \colon A \to B$ such that
\[
f\mu_A = \mu_B f^{\otimes 2}.
\]
A \textbf{morphism} $f \colon A \to B$ is a weak morphism such that
\[
f\alpha_A = \alpha_Bf.
\]
\end{enumerate}
\end{definition}

From now on, an algebra $(A,\mu)$ is also regarded as a Hom-algebra $(A,\mu,Id)$ with identity twisting map.

Next we recall the Hom-type generalizations of (left/right) alternative and flexible algebras.

\begin{definition}
\label{def:homassociator}
Let $(A,\mu,\alpha)$ be a Hom-algebra.
\begin{enumerate}
\item
Define the \textbf{Hom-associator} \cite{ms} $as_A \colon A^{\otimes 3} \to A$ by
\begin{equation}
\label{homassociator}
as_A = \mu (\mu \otimes \alpha - \alpha \otimes \mu).
\end{equation}
\item
$A$ is called a \textbf{right Hom-alternative algebra} if
\begin{equation}
\label{rhomalt}
as_A(x,y,y) = 0
\end{equation}
for all $x,y \in A$.  $A$ is called a \textbf{left Hom-alternative algebra} if
\[
as_A(x,x,y) = 0
\]
for all $x,y \in A$.  $A$ is called a \textbf{Hom-alternative algebra} \cite{mak} if it is both left Hom-alternative and right Hom-alternative.
\item
$A$ is called a \textbf{Hom-flexible algebra} \cite{ms} if
\[
as_A(x,y,x) = 0
\]
for all $x,y \in A$.
\end{enumerate}
\end{definition}

In terms of elements $x,y,z \in A$, we have
\begin{equation}
\label{homass}
as_A(x,y,z) = (xy)\alpha(z) - \alpha(x)(yz)
\end{equation}
When there is no danger of confusion, we will omit the subscript $A$ in the Hom-associator.  When the twisting map $\alpha$ is the identity map, the above definitions reduce to the usual definitions of the associator, (left/right) alternative algebras, and flexible algebras.

It is well-known that alternative algebras are both Jordan-admissible and Maltsev-admissible \cite{bk,maltsev}.  Moreover, alternative algebras satisfy the Moufang identities \cite{schafer}.  Generalizations of these results to multiplicative Hom-alternative algebras are proved in \cite{yau4}.

The following basic result gives the linearized form of the right Hom-alternative identity.

\begin{lemma}[\cite{mak}]
\label{lem:linearized}
Let $(A,\mu,\alpha)$ be a Hom-algebra.  Then the following statements are equivalent.
\begin{enumerate}
\item
$A$ is right Hom-alternative.
\item
$A$ satisfies
\begin{equation}
\label{righthomalt}
as_A(x,y,z) = - as_A(x,z,y)
\end{equation}
for all $x,y,z \in A$.
\item
$A$ satisfies
\begin{equation}
\label{righthomalt'}
(xy)\alpha(z) + (xz)\alpha(y) = \alpha(x)(yz + zy)
\end{equation}
for all $x,y,z \in A$.
\end{enumerate}
\end{lemma}

\begin{proof}
The equivalence of the first two statements is part of \cite{mak} (Proposition 2.6).  Indeed, starting from the right Hom-alternative identity \eqref{rhomalt}, one replaces $y$ by $y + z$ to obtain \eqref{righthomalt}.  Conversely, starting from \eqref{righthomalt}, one sets $y = z$ to obtain the right Hom-alternative identity.  The identity \eqref{righthomalt'} is the expansion of \eqref{righthomalt} using \eqref{homass}.
\end{proof}

A Hom-alternative algebra is right Hom-alternative by definition.  The following observation gives a necessary and sufficient condition under which the converse is true.  This observation is a slight extension of \cite{mak} (Proposition 2.10) and is the Hom-version of \cite{albert3} (Lemma 2).

\begin{lemma}
\label{lem:lefthomalt}
A right Hom-alternative algebra is Hom-alternative if and only if it is Hom-flexible.
\end{lemma}

\begin{proof}
Let $(A,\mu,\alpha)$ be a right Hom-alternative algebra. Setting $y = x$ in \eqref{righthomalt}, we have
\[
as(x,x,z) = -as(x,z,x)
\]
for all $x,z \in A$.  Therefore, in a right Hom-alternative algebra, the left Hom-alternative identity is equivalent to the Hom-flexible identity.
\end{proof}

Let us now observe that the category (with weak morphisms) of right Hom-alternative algebras is closed under twisting by self-weak morphisms.

\begin{theorem}
\label{thm:twist}
Let $(A,\mu,\alpha)$ be a right Hom-alternative algebra, and let $\beta \colon A \to A$ be a weak morphism.  Then the Hom-algebra
\begin{equation}
\label{abeta}
A_\beta = (A,\mubeta = \beta \mu, \beta\alpha)
\end{equation}
is also a right Hom-alternative algebra.  Moreover, if $A$ is multiplicative and $\beta$ is a morphism, then $A_\beta$ is multiplicative.
\end{theorem}

\begin{proof}
In fact, given \emph{any} Hom-algebra $A$ and a weak morphism $\beta \colon A \to A$, we have
\[
\begin{split}
\beta^2 as_A
&= (\beta\mu) (\beta\mu \otimes \beta\alpha - \beta\alpha \otimes \beta\mu)\\
&= as_{A_\beta}.
\end{split}
\]
This implies that, if $A$ is right Hom-alternative, then so is $A_\beta$ .  For the second assertion, assume that $A$ is multiplicative and that $\beta$ is a morphism.  Then we have
\[
\begin{split}
(\beta\alpha)\mubeta &= \beta\alpha\beta\mu\\
&= \beta \mu \alpha^{\otimes 2} \beta^{\otimes 2}\\
&= \mubeta (\alpha\beta)^{\otimes 2}\\
&= \mubeta (\beta\alpha)^{\otimes 2}.
\end{split}
\]
This shows that $A_\beta$ is multiplicative.
\end{proof}

Two special cases of Theorem \ref{thm:twist} follow.  The following result says that each multiplicative right Hom-alternative algebra gives rise to a derived sequence of multiplicative right Hom-alternative algebras.

\begin{corollary}
\label{cor1:twist}
Let $(A,\mu,\alpha)$ be a multiplicative right Hom-alternative algebra.  Then
\[
A^n = (A,\mun = \alpha^{n}\mu,\alpha^{n+1})
\]
is a multiplicative right Hom-alternative algebra for each $n \geq 0$.
\end{corollary}

\begin{proof}
The multiplicativity of $A$ implies that $\alpha^n \colon A \to A$ is a morphism.  By Theorem \ref{thm:twist} $A_{\alpha^n} = A^n$ is a multiplicative right Hom-alternative algebra.
\end{proof}

The following special case of Theorem \ref{thm:twist} says that multiplicative right Hom-alternative algebras may arise from right alternative algebras and their morphisms.  A twisting construction result of this form was first given by the author in \cite{yau1} for $G$-Hom-associative algebras, which include Hom-associative and Hom-Lie algebras.  The following adaptation to right Hom-alternative algebras appeared in \cite{mak} (Theorem 3.1).

\begin{corollary}
\label{cor2:twist}
Let $(A,\mu)$ be a right alternative algebra and $\beta \colon A \to A$ be an algebra morphism.  Then
\[
A_\beta = (A,\mubeta = \beta\mu,\beta)
\]
is a multiplicative right Hom-alternative algebra.
\end{corollary}

\begin{proof}
This is the $\alpha = Id$ case of Theorem \ref{thm:twist}.
\end{proof}

Using Corollary \ref{cor2:twist} we now construct an infinite family of distinct isomorphism classes of multiplicative right Hom-alternative algebras that are neither left Hom-alternative nor right alternative.

\begin{example}
\label{ex:5d}
In \cite{albert3} (p.\ 320-321) Albert constructed a five-dimensional right alternative algebra $(A,\mu)$ that is not left alternative.  In terms of a basis $\{e,u,v,w,z\}$ of $A$, its multiplication $\mu$ is given by
\[
e^2 = e,\quad eu = v,\quad ue = u,\quad
ew = w - z,\quad ez = z = ze,
\]
where the unspecified products of the basis elements are all $0$.

Let $\gamma, \delta, \epsilon \in \bk$ be arbitrary scalars with $\delta \not= 0,1$.  Consider the linear map $\alpha_{\gamma, \delta, \epsilon} = \alpha \colon A \to A$ given by
\begin{equation}
\label{5dalpha}
\alpha(e) = e + \epsilon u + \epsilon v,\quad \alpha(u) = \delta u,\quad
\alpha(v) = \delta v,\quad \alpha(w) = \gamma w,\quad \alpha(z) = \gamma z.
\end{equation}
We claim that $\alpha$ is an algebra morphism on $A$.  Indeed, suppose
\[
x = \lambda_1 e + \lambda_2 u + \lambda_3 v + \lambda_4 w + \lambda_5 z,\quad
y = \theta_1 e + \theta_2 u + \theta_3 v + \theta_4 w + \theta_5 z
\]
are two arbitrary elements in $A$ with all $\lambda_i, \theta_j \in \bk$.  Then
\[
\begin{split}
xy &= \lambda_1\theta_1 e + \lambda_2\theta_1 u + \lambda_1\theta_2 v + \lambda_1\theta_4 w + (\lambda_1(\theta_5 - \theta_4) + \lambda_5\theta_1)z,\\
\alpha(x) &= \lambda_1 e + (\lambda_1\epsilon + \lambda_2\delta)u + (\lambda_1\epsilon + \lambda_3\delta)v + \lambda_4\gamma w + \lambda_5\gamma z,
\end{split}
\]
and similarly for $\alpha(y)$.  A quick computation then shows that
\[
\begin{split}
\alpha(xy)
&= \lambda_1\theta_1 e + \theta_1(\lambda_1 \epsilon + \lambda_2\delta)u + \lambda_1(\theta_1\epsilon + \theta_2\delta)v\\
&\relphantom{} + \lambda_1\theta_4\gamma w + \gamma(\lambda_1(\theta_5 - \theta_4) + \lambda_5\theta_1)z\\
&= \alpha(x)\alpha(y).
\end{split}
\]
By Corollary \ref{cor2:twist} there is a multiplicative right Hom-alternative algebra
\begin{equation}
\label{5daalpha}
A_\alpha = (A,\mualpha = \alpha\mu,\alpha).
\end{equation}
We now prove the following statements.
\begin{enumerate}
\item
$A_\alpha$ in \eqref{5daalpha} is not left Hom-alternative.
\item
$(A,\mualpha)$ is not right alternative.
\item
If $\alpha' = \alpha_{\gamma', \delta', \epsilon'} \colon A \to A$ corresponds to the scalars $\gamma', \delta', \epsilon'$ such that $\delta \not\in \{\delta',\gamma'\}$, then $A_\alpha$ and $A_{\alpha'}$ are not isomorphic as Hom-modules (and hence as Hom-algebras).
\end{enumerate}
To see that $A_\alpha$ is not left Hom-alternative, observe that
\[
\begin{split}
as_{A_\alpha}(e,e,u) &= \alpha^2 as_A(e,e,u)\\
&= \alpha^2((ee)u - e(eu))\\
&= \alpha^2(v)\\
&= \delta^2v,
\end{split}
\]
which is not $0$ because $\delta \not= 0$.  To see that $(A,\mualpha)$ is not right alternative, observe that
\[
\begin{split}
\mualpha(\mualpha(u,e),e) - \mualpha(u,\mualpha(e,e))
&= \alpha\mu(\delta u,e) - \alpha\mu(u, e + \epsilon u + \epsilon v)\\
&= (\delta^2 - \delta)u,
\end{split}
\]
which is not $0$ because $\delta \not= 0,1$.

To prove the last statement, assume to the contrary that there is a Hom-module isomorphism $f \colon A_\alpha \to A_{\alpha'}$.  In particular, we have
\[
f(u) = \lambda_1 e + \lambda_2 u + \lambda_3 v + \lambda_4 w + \lambda_5 z
\]
for some scalars $\lambda_i$, not all of which are $0$.  So we have
\begin{equation}
\label{alphaf}
\alpha'f(u) = \lambda_1 (e + \epsilon' u + \epsilon' v) + \delta'(\lambda_2 u + \lambda_3 v) + \gamma'(\lambda_4 w + \lambda_5 z).
\end{equation}
On the other hand, we have
\begin{equation}
\label{falpha}
f\alpha(u) = f(\delta u) = \delta(\lambda_1 e + \lambda_2 u + \lambda_3 v + \lambda_4 w + \lambda_5 z).
\end{equation}
Since $\alpha'f = f\alpha$ and $\delta \not= 1$, by comparing the coefficients of $e$ in \eqref{alphaf} and \eqref{falpha}, we infer that
\[
\lambda_1 = 0.
\]
Using this in \eqref{alphaf} and \eqref{falpha} and the assumption $\delta \not\in \{\delta', \gamma'\}$, we further infer that
\[
\lambda_i = 0
\]
for $2 \leq i \leq 5$.  We conclude that
\[
f(u) = 0,
\]
contradicting the assumption that $f$ is a linear isomorphism.  Therefore, $A_\alpha$ and $A_{\alpha'}$ are not isomorphic as Hom-modules.
\qed
\end{example}

\section{Hom-power associativity}
\label{sec:hompower}

The purpose of this section is to show that every multiplicative right Hom-alternative algebra is Hom-power associative.  A more general result was proved in \cite{yau15}.  The point of the following proof is that, at least for multiplicative right Hom-alternative algebras, Hom-power associativity can be established by a short and direct proof.

Let us first recall the relevant definitions from \cite{yau15}.

\begin{definition}
\label{def:hompower}
Let $(A,\mu,\alpha)$ be a Hom-algebra, $x \in A$, and $n$ be a positive integer.
\begin{enumerate}
\item
Define the \textbf{$n$th Hom-power} $x^n \in A$ inductively by
\begin{equation}
\label{hompower}
x^1 = x, \qquad
x^n = x^{n-1}\alpha^{n-2}(x)
\end{equation}
for $n \geq 2$.
\item
For positive integers $i$ and $j$, define
\begin{equation}
\label{xij}
x^{i,j} = \alpha^{j-1}(x^i) \alpha^{i-1}(x^j).
\end{equation}
\item
We say that $A$ is \textbf{$n$th Hom-power associative} if
\begin{equation}
\label{nhpa}
x^n = x^{n-i,i}
\end{equation}
for all $x \in A$ and $i \in \{1,\ldots, n-1\}$.
\item
We say that $A$ is \textbf{Hom-power associative} if $A$ is $n$th Hom-power associative for all $n \geq 2$.
\end{enumerate}
\end{definition}

Note that by definition
\[
x^n = x^{n-1,1}
\]
for all $n \geq 2$.

If the twisting map $\alpha$ is the identity map, then
\[
x^n = x^{n-1}x, \quad x^{i,j} = x^ix^j,
\]
and $n$th Hom-power associativity reduces to
\begin{equation}
\label{npa}
x^n = x^{n-i}x^i
\end{equation}
for all $x \in A$ and $i \in \{1,\ldots, n-1\}$.  Therefore,
Hom-powers and ($n$th) Hom-power associativity become Albert's right powers and ($n$th) power associativity \cite{albert1,albert2} if $\alpha = Id$.  Examples of and construction results for Hom-power associative algebras can be found in \cite{yau15}.

A well-known result of Albert \cite{albert1} says that (in characteristic $0$, which is assumed throughout this paper) an algebra $(A,\mu)$ is power associative if and only if it is third and fourth power associative, i.e., the condition \eqref{npa} holds for $n = 3$ and $4$.  Moreover, for \eqref{npa} to hold for $n = 3$ and $4$, it is necessary and sufficient that
\[
(xx)x = x(xx) \quad\text{and}\quad ((xx)x)x = (xx)(xx)
\]
for all $x \in A$.  This result of Albert is remarkable because it shows that power associativity, which has infinitely many defining identities (namely, \eqref{npa} for all $n$), is implied by just two identities.  The Hom-versions of these statements are also true.  More precisely, the author proved in \cite{yau15} that a multiplicative Hom-algebra $(A,\mu,\alpha)$ is Hom-power associative if and only if it is third and fourth Hom-power associative, which in turn is equivalent to
\begin{equation}
\label{34}
x^2\alpha(x) = \alpha(x)x^2 \quad\text{and}\quad
x^4 = \alpha(x^2)\alpha(x^2)
\end{equation}
for all $x \in A$.

It is easy to verify that the identities in \eqref{34} hold in every multiplicative right Hom-alternative algebra (\cite{yau15} Proposition 3.4).  It follows that multiplicative right Hom-alternative algebras are Hom-power associative.  We now give a more direct proof of this result, which is modeled after \cite{albert3} (Lemma 1), without relying on the results from \cite{yau15}.

\begin{theorem}
\label{thm:power}
Every multiplicative right Hom-alternative algebra is Hom-power associative.
\end{theorem}

\begin{proof}
Let $(A,\mu,\alpha)$ be a multiplicative right Hom-alternative algebra.  We prove the $n$th Hom-power associativity \eqref{nhpa} of $A$ by induction on $n \geq 2$.  Pick an element $x$ in $A$.  The case $n=2$ of \eqref{nhpa} is trivially true, since
\[
x^2 = xx = x^{1,1}
\]
by definition.

Inductively suppose we already proved that $A$ is $k$th Hom-power associative for all $k$ in the range $\{2,\ldots,n-1\}$ for some $n \geq 3$.  We establish the $n$th Hom-power associativity \eqref{nhpa} of $A$ by induction on $i$.  Suppose $i \in \{1,\ldots,n-2\}$.  In \eqref{righthomalt'} replace $(x,y,z)$ by
\begin{equation}
\label{xyz}
\left(\alpha^{i-1}(x^{n-(i+1)}), \alpha^{n-i-2}(x^i), \alpha^{n-3}(x)\right).
\end{equation}
Using the induction hypothesis, the left-hand side of \eqref{righthomalt'} then becomes:
\begin{equation}
\label{lpower}
\begin{split}
& x^{n-(i+1),i}\alpha^{n-2}(x) + \left(\alpha^{i-1}(x^{n-(i+1)})\alpha^{n-3}(x)\right) \alpha^{n-i-1}(x^i)\\
&= x^{n-1}\alpha^{n-2}(x) + \alpha^{i-1}\left(x^{n-(i+1)}\alpha^{n-i-2}(x)\right) \alpha^{n-i-1}(x^i)\\
&= x^n + \alpha^{i-1}(x^{n-i}) \alpha^{n-i-1}(x^i)\\
&= x^n + x^{n-i,i}.
\end{split}
\end{equation}
Likewise, by the induction hypothesis, the right-hand side of \eqref{righthomalt'} becomes:
\begin{equation}
\label{rpower}
\begin{split}
& \alpha^i(x^{n-(i+1)}) \alpha^{n-i-2}\left(x^i\alpha^{i-1}(x) + \alpha^{i-1}(x)x^i\right)\\
&= \alpha^i(x^{n-(i+1)}) \alpha^{n-i-2}\left(x^{i+1} + x^{i+1}\right)\\
&= 2\alpha^i(x^{n-(i+1)}) \alpha^{n-i-2}(x^{i+1})\\
&= 2x^{n-(i+1),i+1}.
\end{split}
\end{equation}
Therefore, the special case of \eqref{righthomalt'} with $(x,y,z)$ as in \eqref{xyz} says
\begin{equation}
\label{powerind}
2x^{n-(i+1),i+1} = x^n + x^{n-i,i}
\end{equation}
for $i \in \{1,\ldots,n-2\}$.  It follows from the definition $x^n = x^{n-1,1}$ and \eqref{powerind} that
\[
x^n = x^{n-i,i}
\]
for $i \in \{1,\ldots,n-1\}$, so $A$ is $n$th Hom-power associative.
\end{proof}

Setting $\alpha = Id$ in Theorem \ref{thm:power}, we recover \cite{albert3} (Lemma 1).

\begin{corollary}
\label{cor:power}
Every right alternative algebra is power associative.
\end{corollary}

\section{Hom-Jordan admissibility}
\label{sec:homjordan}

The purpose of this section is to show that multiplicative right Hom-alternative algebras are Hom-Jordan admissible.  It is well-known that alternative algebras are Jordan admissible. The Hom-version of this statement is also true.  More precisely, the author proved in \cite{yau4} that every multiplicative Hom-alternative algebra is Hom-Jordan admissible.  Here we strengthen this result by removing the left Hom-alternativity assumption, while at the same time simplifying the argument.

Let us first recall some relevant definitions from \cite{yau4}.

\begin{definition}
\label{def:homjordan}
Let $(A,\mu,\alpha)$ be a Hom-algebra.
\begin{enumerate}
\item
Define the \textbf{plus Hom-algebra}
\[
A^+ = (A,\ast,\alpha),
\]
where $\ast = (\mu + \muop)/2$.
\item
$A$ is called a \textbf{Hom-Jordan algebra} if $\mu = \muop$ (commutativity) and the \textbf{Hom-Jordan identity}
\[
as_A(x^2,\alpha(y),\alpha(x)) = 0
\]
is satisfied for all $x,y \in A$, where $as_A$ is the Hom-associator \eqref{homassociator}.
\item
$A$ is called a \textbf{Hom-Jordan admissible algebra} if its plus Hom-algebra $A^+$ is a Hom-Jordan algebra.
\end{enumerate}
\end{definition}

If $\alpha = Id$ then the above definitions reduce to the usual notions of plus algebras and Jordan (admissible) algebras.

The reader is cautioned that the above definition of a Hom-Jordan algebra is not the same as the one in \cite{mak}.  A Hom-Jordan algebra in the sense of \cite{mak} is also a Hom-Jordan algebra in the sense of Definition \ref{def:homjordan}, but the converse is not true.

Note that the Jordan product
\[
x \ast y = \frac{1}{2}(\mu(x,y) + \mu(y,x)) = \frac{1}{2}(xy + yx) = y \ast x
\]
is commutative and that
\[
x \ast x = \mu(x,x) = x^2
\]
for all $x \in A$.

\begin{lemma}
\label{lem:jordan}
Let $(A,\mu,\alpha)$ be a Hom-algebra.  Then the following statements are equivalent.
\begin{enumerate}
\item
$A$ is Hom-Jordan admissible.
\item
The condition
\begin{equation}
\label{jhas}
as_{A^+}(x^2,\alpha(y),\alpha(x)) = 0
\end{equation}
holds for all $x,y \in A$.
\item
The condition
\begin{equation}
\label{hjplus}
(\alpha(x) \ast \alpha(y)) \ast \alpha(x^2) - \alpha^2(x) \ast (\alpha(y) \ast x^2) = 0
\end{equation}
holds for all $x,y \in A$.
\end{enumerate}
\end{lemma}

\begin{proof}
The equivalence of the first two statements follows from the commutativity of $\ast$ and the identity $x \ast x = x^2$.  The equivalence of \eqref{jhas} and \eqref{hjplus} follows by expanding the Hom-associator $as_{A^+}$ in \eqref{jhas} and using the commutativity of $\ast$.
\end{proof}

We are now ready for the main result of this section, which is the Hom-version of Theorem 2 in \cite{albert3}.

\begin{theorem}
\label{thm:jordan}
Every multiplicative right Hom-alternative algebra is Hom-Jordan admissible
\end{theorem}

\begin{proof}
Let $(A,\mu,\alpha)$ be a multiplicative right Hom-alternative algebra.  By Lemma \ref{lem:jordan} it suffices to prove \eqref{hjplus}.  Pick elements $x,y \in A$.  First note that the right Hom-alternative identity \eqref{rhomalt} implies
\begin{equation}
\label{thirdhpa}
0 = as(x,x,x) = x^2\alpha(x) - \alpha(x)x^2.
\end{equation}
Also, the special cases of \eqref{righthomalt'} with $(x,y,z)$ replaced by $(x^2,\alpha(x),\alpha(y))$ and $(\alpha(x),\alpha(y),x^2)$ are
\begin{equation}
\label{xyz1}
\alpha(x^2) \left(\alpha(x)\alpha(y) + \alpha(y)\alpha(x)\right) = \left(x^2\alpha(x)\right)\alpha^2(y) + \left(x^2\alpha(y)\right)\alpha^2(x)
\end{equation}
and
\begin{equation}
\label{xyz2}
\alpha^2(x) \left(\alpha(y)x^2 + x^2\alpha(y)\right) = \left(\alpha(x)\alpha(y)\right)\alpha(x^2) + \left(\alpha(x)x^2\right)\alpha^2(y),
\end{equation}
respectively.

Using \eqref{thirdhpa}, \eqref{xyz1}, and \eqref{xyz2}, we now compute the left-hand side of \eqref{hjplus} multiplied by $4$:
\[
\begin{split}
& 4\left(\alpha(x) \ast \alpha(y)\right) \ast \alpha(x^2) - 4\alpha^2(x) \ast \left(\alpha(y) \ast x^2\right)\\
&= \left(\alpha(x)\alpha(y) + \alpha(y)\alpha(x)\right) \alpha(x^2) + \alpha(x^2) \left(\alpha(x)\alpha(y) + \alpha(y)\alpha(x)\right)\\
&\relphantom{} - \alpha^2(x) \left(\alpha(y)x^2 + x^2\alpha(y)\right) - \left(\alpha(y)x^2 + x^2\alpha(y)\right) \alpha^2(x)\\
&= \left(\alpha(x)\alpha(y)\right) \alpha(x^2) + \underbrace{\left(\alpha(y)\alpha(x)\right) \alpha(x^2)}_{(a)} + \left(x^2\alpha(x)\right)\alpha^2(y) + \left(x^2\alpha(y)\right)\alpha^2(x)\\
&\relphantom{} - \left(\alpha(x)\alpha(y)\right)\alpha(x^2) - \left(\alpha(x)x^2\right)\alpha^2(y) - \underbrace{\left(\alpha(y)x^2\right) \alpha^2(x)}_{(b)} - \left(x^2\alpha(y)\right) \alpha^2(x)\\
&= \underbrace{\alpha(yx)\left(\alpha(x)\alpha(x)\right)}_{(a)} - \underbrace{((yx)\alpha(x))\alpha^2(x)}_{(b)}\\
&= -as_A(yx,\alpha(x),\alpha(x)) = 0.
\end{split}
\]
This shows that \eqref{hjplus} is satisfied, so $A$ is Hom-Jordan admissible.
\end{proof}

Setting $\alpha = Id$ in Theorem \ref{thm:jordan}, we recover \cite{albert3} (Theorem 2).

\begin{corollary}
\label{cor:jordan}
Every right alternative algebra is Jordan admissible.
\end{corollary}

\section{Albert-type decompositions}
\label{sec:decomp}

The purpose of this section is to study Albert-type decompositions for right Hom-alternative algebras with respect to idempotents.

Let us first define the Hom-version of an idempotent.

\begin{definition}
\label{def:idempotent}
Let $(A,\mu,\alpha)$ be a Hom-algebra.  An \textbf{idempotent} in $A$ is an element $e \in A$ that satisfies
\begin{equation}
\label{idempotent}
e^2 = e = \alpha(e),
\end{equation}
where $e^2 = \mu(e,e)$.
\end{definition}

In the above definition, if $\alpha = Id$ then $e \in A$ is an idempotent if and only if $e^2 = e$.  This, of course, is the usual definition of an idempotent in an algebra.

We first observe that the property of being an idempotent is preserved under the constructions in section \ref{sec:ex}.  The following result is an immediate consequence of Definition \ref{def:idempotent}.

\begin{proposition}
\label{prop1:idempotent}
Let $(A,\mu,\alpha)$ be a Hom-algebra and $e \in A$ be an idempotent.
\begin{enumerate}
\item
If $\beta \colon A \to A$ is a linear map such that $\beta(e) = e$, then $e$ is an idempotent in
\[
A_\beta = (A,\mubeta = \beta\mu,\beta\alpha).
\]
\item
The element $e$ is an idempotent in
\[
A^n = (A,\mun = \alpha^{n}\mu,\alpha^{n+1})
\]
for each $n \geq 0$.
\end{enumerate}
\end{proposition}

The following observation is the $\alpha = Id$ special case of the first part of Proposition \ref{prop1:idempotent}.

\begin{corollary}
\label{cor1:idempotent}
Let $(A,\mu)$ be an algebra, $e \in A$ be an idempotent, and $\beta \colon A \to A$ be a linear map such that $\beta(e) = e$.  Then $e$ is an idempotent in $A_\beta = (A,\mubeta = \beta\mu,\beta)$.
\end{corollary}

\begin{example}
In the five-dimensional right alternative algebra $A$ in Example \ref{ex:5d}, the basis element $e$ is an idempotent.  The map $\alpha_{\gamma, \delta, \epsilon} = \alpha \colon A \to A$ in \eqref{5dalpha} fixes $e$ if and only if $\epsilon = 0$.  Therefore, by Corollary \ref{cor1:idempotent} the element $e$ is an idempotent in the multiplicative right Hom-alternative algebra $A_\alpha$ in \eqref{5daalpha} if and only if $\epsilon = 0$.
\qed
\end{example}

Let $e$ be an idempotent in a right alternative algebra $A$.  Then there is an \textbf{Albert decomposition} \cite{albert3}
\begin{equation}
\label{albertdecomp}
A = A_e(1) \oplus A_e(0),
\end{equation}
where
\[
A_e(i) = \{a \in A \colon ae = ia\}
\]
for $i=0,1$.  More precisely, every $a \in A$ can be decomposed uniquely as
\[
a = ae + (a - ae)
\]
with $ae \in A_e(1)$ and $(a - ae) \in A_e(0)$.

The following result is a Hom-type generalization of the Albert decomposition \eqref{albertdecomp} for right Hom-alternative algebras.

\begin{proposition}
\label{prop:decomp}
Let $(A,\mu,\alpha)$ be a multiplicative right Hom-alternative algebra in which $\alpha$ is surjective and $e \in A$ be an idempotent.  Then there is a not-necessarily-direct sum
\begin{equation}
\label{decompabar}
A = A_e(\alpha) + A_e(0)
\end{equation}
of sub-Hom-modules, where
\[
A_e(i\alpha) = \{a \in A \colon ae = i\alpha(a)\}
\]
for $i\in\{0,1\}$.  The sum \eqref{decompabar} is a direct sum if $\alpha$ is also injective.
\end{proposition}

\begin{proof}
First observe that, if $\alpha$ is injective, then the intersection of $A_e(\alpha)$ and $A_e(0)$ is $0$.  Indeed, if $x \in A$ belongs to both submodules, then
\[
\alpha(x) = xe = 0.
\]
The injectivity of $\alpha$ implies that $x = 0$.

To see that the submodule $A_e(i\alpha)$ is a sub-Hom-module for $i\in \{0,1\}$, suppose $a \in A_e(i\alpha)$.  Then
\[
\begin{split}
\alpha(a)e &= \alpha(a)\alpha(e)\\
&= \alpha(ae)\\
&= \alpha(i\alpha(a))\\
&= i\alpha(\alpha(a)).
\end{split}
\]
So $\alpha(a)$ is also in $A_e(i\alpha)$.

It remains to show that the sum $A_e(\alpha) + A_e(0)$ is all of $A$.  Pick an arbitrary element $b \in A$.  The surjectivity of $\alpha$ implies that $b = \alpha(a)$ for some $a \in A$.  Consider the decomposition
\[
b = ae + (b - ae).
\]
We claim that $ae \in A_e(\alpha)$ and that $(b - ae) \in A_e(0)$.  Indeed, we have
\[
\begin{split}
(ae)e &= (ae)\alpha(e)\\
&= \alpha(a)(ee)\\
&= \alpha(a)\alpha(e)\\
&= \alpha(ae).
\end{split}
\]
This shows that $ae \in A_e(\alpha)$.  Likewise, we have
\[
\begin{split}
(b - ae)e &= be - (ae)\alpha(e)\\
&= \alpha(a)e - \alpha(a)(ee)\\
&= 0,
\end{split}
\]
which shows that $(b - ae) \in A_e(0)$.
\end{proof}

In particular, if $\alpha = Id$ in Proposition \ref{prop:decomp}, then $A_e(Id) = A_e(1)$ and the decomposition $A_e(Id) \oplus A_e(0)$ becomes the Albert decomposition \eqref{albertdecomp}.

\section{Multiplication operators induced by idempotents}
\label{sec:idempotent}

In this section, we study multiplication operators in right Hom-alternative algebras, especially those induced by idempotents.  Let us first fix some notations.

\begin{definition}
\label{def:multiplication}
Let $(A,\mu,\alpha)$ be a Hom-algebra and $x$ be an element in $A$.  Define $L_x, R_x \colon A \to A$ to be the operators of left and right multiplication by $x$, acting from the right, i.e.,
\[
aL_x = xa \quad\text{and}\quad aR_x = ax
\]
for all $a \in A$.
\end{definition}

The convention of $L_x$ and $R_x$ acting from the right is often used in the literature, e.g., \cite{albert2,albert3}.  We also allow the twisting map $\alpha \colon A \to A$ to act from the right, i.e.,
\[
a\alpha = \alpha(a)
\]
for $a \in A$.  The operators $\alpha$, $L_x$, and $R_x$ are all in $\Hom(A,A)$, the module of linear operators on $A$ acting from the right.  Since composition of linear operators is associative, $\Hom(A,A)$ is a Lie algebra under the commutator bracket,
\[
[f,g] = fg - gf
\]
for $f,g \in \Hom(A,A)$.

\begin{lemma}
\label{lem:Rx}
Let $(A,\mu,\alpha)$ be a right Hom-alternative algebra.  Then
\begin{equation}
\label{righthomalt2}
R_x R_{\alpha(x)} = \alpha R_{xx}
\end{equation}
and
\begin{equation}
\label{righthomalt3}
L_y L_{\alpha(x)} - \alpha L_{xy} = L_x R_{\alpha(y)} - R_y L_{\alpha(x)}
\end{equation}
for all $x,y \in A$.
\end{lemma}

\begin{proof}
The desired identities \eqref{righthomalt2} and \eqref{righthomalt3} are the operator forms of the right Hom-alternative identity \eqref{rhomalt} and the linearized right Hom-alternative identity \eqref{righthomalt'}, respectively.
\end{proof}

We now restrict our attention to the multiplication operators induced by idempotents.

\begin{proposition}
\label{lem1:idempotent}
Let $(A,\mu,\alpha)$ be a right Hom-alternative algebra and $e \in A$ be an idempotent.  Then
\begin{equation}
\label{re}
R_e^{n+1} = \alpha^n R_e
\end{equation}
for all $n \geq 0$, and
\begin{equation}
\label{le}
L_e^2 - \alpha L_e = [L_e, R_e].
\end{equation}
If, in addition, $A$ is multiplicative, then
\begin{equation}
\label{lr}
[\alpha, L_e] = 0 = [\alpha, R_e]
\end{equation}
\end{proposition}

\begin{proof}
Write $R = R_e$.  The condition \eqref{re} is clearly true for $n = 0$.  Inductively, suppose \eqref{re} is true for a particular $n$.  Then we have
\[
\begin{split}
R^{n+2} &= R^{n+1}R\\
&= \alpha^n RR\quad\text{(by induction hypothesis)}\\
&= \alpha^n \alpha R,
\end{split}
\]
where the last equality follows from \eqref{righthomalt2} because $\alpha(e) = e = e^2$.  Therefore, \eqref{re} is true for all $n$.

The equality \eqref{le} is the special case of \eqref{righthomalt3} with $x = y = e$.

For the last assertion, the multiplicativity of $A$ implies
\[
L_x \alpha = \alpha L_{\alpha(x)} \quad\text{and}\quad R_x \alpha = \alpha R_{\alpha(x)}
\]
for all $x \in A$.  Restricting to the special case $x = e$, we obtain \eqref{lr}.
\end{proof}

A linear self-map $f \colon V \to V$ on a module $V$ is said to be an \textbf{idempotent} if $f^2 = f$.  We now introduce the Hom-version of this concept.

\begin{definition}
\label{def:alphaidempotent}
Let $(A,\mu,\alpha)$ be a Hom-algebra, $f \colon A \to A$ be a linear map acting from the right, and $n$ be an integer.  We say that $f$ is an \textbf{$\alpha^n$-idempotent} if
\[
f^2 = \alpha^n f
\]
in $\Hom(A,A)$.
\end{definition}

In particular, \eqref{re} with $n = 1$ says
\begin{equation}
\label{Rsquare}
R_e^2 = \alpha R_e.
\end{equation}
In other words, in a right Hom-alternative algebra, the right multiplication operator $R_e$ is an $\alpha$-idempotent.  On the other hand, \eqref{le} tells us that the left multiplication operator $L_e$ is not an $\alpha$-idempotent.  The following result, which is the Hom-version of \cite{albert3} (Lemma 4), says that $L_e$ is not too far from being an $\alpha$-idempotent.

\begin{theorem}
\label{prop:L}
Let $(A,\mu,\alpha)$ be a multiplicative right Hom-alternative algebra and $e \in A$ be an idempotent.  Then
\[
(L_e^2 - \alpha L_e)^2 = 0 = ([L_e, R_e])^2.
\]
\end{theorem}

\begin{proof}
By \eqref{le} it suffices to prove
\[
(L_e^2 - \alpha L_e)^2 = 0.
\]
Let us write $L = L_e$ and $R = R_e$.  First note that
\begin{equation}
\label{l2}
\begin{split}
(L^2 - \alpha L)^2 &= L^4 - 2\alpha L^3 + \alpha^2 L^2 \quad\text{(by \eqref{lr})}\\
&= (L^3 - \alpha L^2)L - \alpha(L^3 - \alpha L^2).
\end{split}
\end{equation}
To show that $(L^2 - \alpha L)^2 = 0$, we next compute $(L^3 - \alpha L^2)$ in two ways.  On the one hand, we have
\begin{equation}
\label{l3}
\begin{split}
L^3 - \alpha L^2 &= L(L^2 - \alpha L) \quad\text{(by \eqref{lr})}\\
&= L^2R - LRL \quad\text{(by \eqref{le})}\\
&= (\alpha L + LR - RL)R - LRL \quad\text{(by \eqref{le})}\\
&= 2\alpha LR - RLR - LRL \quad\text{(by \eqref{Rsquare} and \eqref{lr})}.
\end{split}
\end{equation}
On the other hand, we have
\begin{equation}
\label{l3'}
\begin{split}
L^3 - \alpha L^2 &= (L^2 - \alpha L)L\\
&= LRL - RL^2 \quad\text{(by \eqref{le})}\\
&= LRL - R(\alpha L + LR - RL) \quad\text{(by \eqref{le})}\\
&= LRL - RLR - (R\alpha - R^2)L\\
&= LRL - RLR \quad\text{(by \eqref{Rsquare} and \eqref{lr})}.
\end{split}
\end{equation}
Comparing \eqref{l3} and \eqref{l3'}, we obtain
\begin{equation}
\label{lrl}
LRL = \alpha LR.
\end{equation}
Using \eqref{lrl} in \eqref{l3'}, we obtain
\begin{equation}
\label{l3''}
L^3 - \alpha L^2 = \alpha LR - RLR.
\end{equation}
Finally, we have
\begin{equation}
\label{l2'}
\begin{split}
(L^3 - \alpha L^2)L &= \alpha LRL - RLRL \quad\text{(by \eqref{l3''})}\\
&= \alpha^2 LR - \alpha RLR \quad\text{(by \eqref{lrl} and \eqref{lr})}\\
&= \alpha (\alpha LR - RLR)\\
&= \alpha(L^3 - \alpha L^2) \quad\text{(by \eqref{l3''})}.
\end{split}
\end{equation}
Comparing \eqref{l2} and \eqref{l2'}, we conclude that $(L^2 - \alpha L)^2 = 0$.
\end{proof}

The next result shows that, although $L_e$ is not $\alpha$-idempotent, there is an operator of degree $3$ in $L_e$ that is $\alpha^4$-idempotent.  This is the Hom-version of an observation in \cite{albert3} (section 6).

\begin{corollary}
\label{cor1:L}
Let $(A,\mu,\alpha)$ be a multiplicative right Hom-alternative algebra, $e \in A$ be an idempotent, and $L = L_e$.  Then the linear operator
\begin{equation}
\label{T}
T = 3\alpha^2L^2 - 2\alpha L^3
\end{equation}
satisfies
\begin{equation}
\label{Tn}
T^{n+1} = \alpha^{4n}T
\end{equation}
for all $n \geq 0$.  In particular, $T$ is $\alpha^4$-idempotent, i.e.,
\[
T^2 = \alpha^4 T.
\]
\end{corollary}

\begin{proof}
The condition \eqref{Tn} is trivially true when $n = 0$.  When $n=1$, we compute using \eqref{lr} as follows:
\begin{equation}
\label{T2}
\begin{split}
T^2 &= 4\alpha^2 L^6 - 12 \alpha^3 L^5 + 9\alpha^4 L^4\\
&= (4\alpha^2 L^2 - 4\alpha^3 L - 3\alpha^4)(L^4 - 2\alpha L^3 + \alpha^2 L^2) + \alpha^4 T\\
&= (4\alpha^2 L^2 - 4\alpha^3 L - 3\alpha^4)(L^2 - \alpha L)^2 + \alpha^4 T\\
&= \alpha^4 T.
\end{split}
\end{equation}
The last equality holds because $(L^2 - \alpha L)^2 = 0$ by Theorem \ref{prop:L}.  Inductively, suppose \eqref{Tn} is true for a particular $n$.  Then we have:
\[
\begin{split}
T^{n+2} &= T^{n+1}T\\
&= \alpha^{4n}TT\quad\text{(by induction hypothesis)}\\
&= \alpha^{4n}\alpha^4 T\quad\text{(by \eqref{T2})}\\
&= \alpha^{4(n+1)}T.
\end{split}
\]
This finishes the induction and proves \eqref{Tn} for all $n$.
\end{proof}

In the context of Lemma \ref{lem1:idempotent}, the operators $L_e$ and $R_e$ do not commute, as can be seen from \eqref{le}.  Our next result says that the $\alpha^4$-idempotent operator $T$ \eqref{T} does commute with $R_e$.  This result is the Hom-version of \cite{albert3} (Lemma 7).

\begin{corollary}
\label{prop:T}
Let $(A,\mu,\alpha)$ be a multiplicative right Hom-alternative algebra, $e \in A$ be an idempotent, $L = L_e$, and $R = R_e$. Then
\[
[T, R] = 0,
\]
where $T = 3\alpha^2L^2 - 2\alpha L^3$ is the operator in Corollary \ref{cor1:L}.
\end{corollary}

\begin{proof}
First we claim that
\begin{equation}
\label{lk}
k(L^{k+1} - \alpha L^k) = [L^k, R]
\end{equation}
for all $k \geq 1$.  The case $k=1$ is true by \eqref{le}.  Inductively, suppose \eqref{lk} is true for a particular $k$. Then we have:
\[
\begin{split}
k(L^{k+2} - \alpha L^{k+1}) &= kL(L^{k+1} - \alpha L^k)\quad\text{(by \eqref{lr})}\\
&= L[L^k, R] \quad\text{(by induction hypothesis)}\\
&= L^{k+1}R - LRL^k\\
&= L^{k+1}R - (L^2 - \alpha L + RL)L^k \quad\text{(by \eqref{le})}\\
&= [L^{k+1}, R] - (L^{k+2} - \alpha L^{k+1}).
\end{split}
\]
This finishes the induction step and proves \eqref{lk} for all $k$.

Using \eqref{lr} and \eqref{lk}, we compute as follows:
\[
\begin{split}
[T, R] &= (3\alpha^2L^2 - 2\alpha L^3)R - R(3\alpha^2L^2 - 2\alpha L^3)\\
&= 3\alpha^2[L^2, R] - 2\alpha[L^3, R]\\
&= 3\alpha^2(2)(L^3 - \alpha L^2) - 2\alpha(3)(L^4 - \alpha L^3)\\
&= -6\alpha(L^2 - \alpha L)^2.
\end{split}
\]
Since $(L^2 - \alpha L)^2 = 0$ by Theorem \ref{prop:L}, we conclude that $[T, R] = 0$.
\end{proof}

\section{Identities in right Hom-alternative algebras}
\label{sec:identity}

The purpose of this section is to prove the following Hom-type generalizations of some well-known and frequently-used identities in right alternative algebras from \cite{kleinfeld}.

\begin{theorem}
\label{thm:id}
Let $(A,\mu,\alpha)$ be a multiplicative right Hom-alternative algebra.  Then the following identities hold for all $w,x,y,z \in A$.
\begin{subequations}
\begin{align}
as(\alpha(x),\alpha(y),yz) &= as(x,y,z)\alpha^2(y). \label{xyyz}\\
as(\alpha(x),\alpha(w),yz) + as(\alpha(x),\alpha(y),wz) &= as(x,w,z)\alpha^2(y) + as(x,y,z)\alpha^2(w).\label{g}\\
as(wx,\alpha(y),\alpha(z)) + as(\alpha(w),\alpha(x),[y,z]) &= \alpha^2(w)as(x,y,z) + as(w,y,z)\alpha^2(x).\label{h}\\
as(\alpha(x),y^2,\alpha(z)) &= as(\alpha(x),\alpha(y),yz+zy).\label{xyyz2}\\
((xy)\alpha(z))\alpha^2(y) &= \alpha^2(x)((yz)\alpha(y)). \label{moufang}\\
(as(x,y,z)\alpha^2(y))\alpha^3(z) &= \alpha(as(x,y,z)\alpha(zy)).\label{xyzyz}
\end{align}
\end{subequations}
\end{theorem}

Setting $\alpha = Id$ in Theorem \ref{thm:id}, we obtain the corresponding identities in right alternative algebras, all of which can be found in \cite{kleinfeld}.  In \eqref{h} the bracket $[,]$ is the commutator bracket of $\mu$, i.e., $[,] = \mu - \muop$.  The identity \eqref{moufang} is one of the Hom-Moufang identities in \cite{yau4}, in which \eqref{moufang} was established with the additional assumption of left Hom-alternativity.

To simplify the presentation, we adopt the following notations.

\begin{definition}
\label{def:f}
Let $(A,\mu,\alpha)$ be a Hom-algebra, and let $w,x,y,z$ be elements in $A$.  Define
\[
(x,y,z) = as_A(x,y,z),\quad
x' = \alpha(x), \quad x'' = \alpha^2(x),\quad x''' = \alpha^3(x).
\]
Define the function $f \colon A^{\otimes 4} \to A$ by
\begin{equation}
\label{f}
\begin{split}
f(w,x,y,z) &= (wx,y',z') - (w',xy,z') + (w',x',yz)\\
&\relphantom{} - w''(x,y,z) - (w,x,y)z''
\end{split}
\end{equation}
for $w,x,y,z \in A$.
\end{definition}

With these notations, if $A$ is multiplicative, then
\[
(xy)' = x'y',\quad (x,y,z)' = (x',y',z')
\]
The right Hom-alternative identity \eqref{rhomalt} and its linearized form \eqref{righthomalt} now say
\[
(x,y,y) = 0 \quad\text{and}\quad (x,y,z) = - (x,z,y),
\]
respectively.

We need the following preliminary result.

\begin{lemma}
\label{lem:teichmuller}
Let $(A,\mu,\alpha)$ be a multiplicative Hom-algebra.  Then
\begin{equation}
\label{teichmuller}
f(w,x,y,z) = 0
\end{equation}
for all $w,x,y,z \in A$.
\end{lemma}

\begin{proof}
Simply expand the five Hom-associators in $f$ \eqref{f}, and observe that the resulting ten terms add to $0$.
\end{proof}

The $\alpha = Id$ special case of \eqref{teichmuller} is known as the Teichm\"{u}ller identity, which holds in any algebra.  We refer to \eqref{teichmuller} as the \textbf{Hom-Teichm\"{u}ller identity}.

In the rest of this section, $(A,\mu,\alpha)$ denotes a multiplicative right Hom-alternative algebra, and $w,x,y,z$ denote arbitrary elements in $A$.  We now prove the identities in Theorem \ref{thm:id}.

\begin{proof}[Proof of \eqref{xyyz}]
We need to prove
\begin{equation}
\label{xyyz'}
(x',y',yz) = (x,y,z)y''.
\end{equation}
Using \eqref{rhomalt}, \eqref{righthomalt}, and the Hom-Teichm\"{u}ller identity, we have:
\[
\begin{split}
0 &= f(x,y,y,z) - f(x,z,y,y) + f(x,y,z,y)\\
&= \underbrace{(xy,y',z')}_{(c)} - \underbrace{(x',yy,z')}_{(d)} + \underbrace{(x',y',yz)}_{(a)} - \underbrace{x''(y,y,z)}_{(e)} - (x,y,y)z''\\
&\relphantom{} - (xz,y',y') + \underbrace{(x',zy,y')}_{(i)} - \underbrace{(x',z',yy)}_{-(d)} + x''(z,y,y) + \underbrace{(x,z,y)y''}_{-(b)}\\
&\relphantom{} + \underbrace{(xy,z',y')}_{-(c)} - \underbrace{(x',yz,y')}_{-(a)} + \underbrace{(x',y',zy)}_{-(i)} - \underbrace{x''(y,z,y)}_{-(e)} - \underbrace{(x,y,z)y''}_{(b)}\\
&= 2\underbrace{(x',y',yz)}_{(a)} - 2\underbrace{(x,y,z)y''}_{(b)}
\end{split}
\]
Dividing by $2$ in the above computation yields \eqref{xyyz'}.
\end{proof}

\begin{proof}[Proof of \eqref{g}]
We need to prove
\begin{equation}
\label{g'}
g(x,w,y,z) \defn (x',w',yz) + (x',y',wz) - (x,w,z)y'' - (x,y,z)w'' = 0.
\end{equation}
Linearize the identity \eqref{xyyz'} by replacing $y$ by $y+w$. The result is exactly \eqref{g'}.
\end{proof}

\begin{proof}[Proof of \eqref{h}]
We need to prove
\begin{equation}
\label{h'}
h(w,x,y,z) \defn (wx,y',z') + (w',x',[y,z]) - w''(x,y,z) - (w,y,z)x'' = 0.
\end{equation}
Using \eqref{righthomalt}, the Hom-Teichm\"{u}ller identity, and \eqref{g'}, we have:
\[
\begin{split}
0 &= f(w,x,y,z) - g(w,z,x,y)\\
&= (wx,y',z') - \underbrace{(w',xy,z')}_{(a)} + (w',x',yz) - w''(x,y,z) - (w,x,y)z''\\
&\relphantom{} - \underbrace{(w',z',xy)}_{-(a)} - (w',x',zy) + (w,z,y)x'' + (w,x,y)z''\\
&= (wx,y',z') + (w',x',yz) - w''(x,y,z) - (w',x',zy) - (w,y,z)x''\\
&= h(w,x,y,z).
\end{split}
\]
This proves \eqref{h'}.
\end{proof}

\begin{proof}[Proof of \eqref{xyyz2}]
By \eqref{righthomalt} the identity \eqref{xyyz2} is equivalent to
\begin{equation}
\label{xyyz2'}
(x',z',y^2) = (x',yz+zy,y').
\end{equation}
Using \eqref{rhomalt}, \eqref{righthomalt}, and the Hom-Teichm\"{u}ller identity, we have:
\[
\begin{split}
0 &= f(x,z,y,y)\\
&= (xz,y',y') - (x',zy,y') + (x',z',yy) - x''(z,y,y) - (x,z,y)y''\\
&= - (x',zy,y') + (x',z',y^2) + (x,y,z)y''\\
&= - (x',zy,y') + (x',z',y^2) + (x',y',yz)\quad\text{(by \eqref{xyyz'})}\\
&= - (x',zy,y') + (x',z',y^2) - (x',yz,y').
\end{split}
\]
This proves \eqref{xyyz2'}.
\end{proof}

\begin{proof}[Proof of \eqref{moufang}]
We need to prove
\begin{equation}
\label{moufang'}
((xy)z')y'' = x''((yz)y').
\end{equation}
We compute as follows:
\[
\begin{split}
((xy)z')y'' &= (x,y,z)y'' + (x'(yz))y'' \quad\text{(by \eqref{homassociator})}\\
&= (x',y',yz) + (x'(yz))y'' \quad\text{(by \eqref{xyyz'})}\\
&= - (x',yz,y') + (x'(yz))y''\quad\text{(by \eqref{righthomalt})}\\
&= - (x'(yz))y'' + x''((yz)y') + (x'(yz))y''\quad\text{(by \eqref{homassociator})}\\
&= x''((yz)y').
\end{split}
\]
This proves the Hom-Moufang identity \eqref{moufang}.
\end{proof}

\begin{proof}[Proof of \eqref{xyzyz}]
By interchanging $y$ and $z$, the desired identity \eqref{xyzyz} is equivalent to
\[
((x,z,y)z'')y''' = (x,z,y)'(yz)''.
\]
Using the linearized right Hom-alternative identity \eqref{righthomalt} in the previous line, it follows that \eqref{xyzyz} is equivalent to
\begin{equation}
\label{xyzyz'}
((x,y,z)z'')y''' = (x,y,z)'(yz)''.
\end{equation}
To prove \eqref{xyzyz'}, first observe that:
\begin{equation}
\label{xyyzz}
\begin{split}
(x'',y'',y'(zz)) &= (x',y',zz)y''' \quad\text{(by \eqref{xyyz'})}\\
&= (x',yz,z')y''' + (x',zy,z')y''' \quad\text{(by \eqref{xyyz2'})}\\
&= (x',yz,z')y''' - (x',z',zy)y''' \quad\text{(by \eqref{righthomalt})}\\
&= (x',yz,z')y''' - ((x,z,y)z'')y''' \quad\text{(by \eqref{xyyz'})}\\
&= (x',yz,z')y''' + ((x,y,z)z'')y''' \quad\text{(by \eqref{righthomalt})}.\\
\end{split}
\end{equation}
Now using \eqref{g'} and \eqref{rhomalt} we have:
\[
\begin{split}
0 &= g(x',y',yz,z')\\
&= (x'',y'',(yz)z') + (x'',(yz)',y'z') - (x',y',z')(yz)'' - (x',yz,z')y'''\\
&= (x'',y'',y'(zz)) + (x'',y'z',y'z') - (x,y,z)'(yz)'' - (x',yz,z')y'''\\
&= (x'',y'',y'(zz)) - (x,y,z)'(yz)'' - (x',yz,z')y'''\\
&= ((x,y,z)z'')y''' - (x,y,z)'(yz)'' \quad\text{(by \eqref{xyyzz})}.
\end{split}
\]
This proves the identity \eqref{xyzyz'}.
\end{proof}



\begin{thebibliography}{AA}
\bibitem{albert1}
A.A. Albert, On the power-associativity of rings, Summa Brasil. Math. 2 (1948) 21-32.

\bibitem{albert2}
A.A. Albert, Power-associative rings, Trans. Amer. Math. Soc. 64 (1948) 552-593.

\bibitem{albert3}
A.A. Albert, On right alternative algebras, Ann. Math. 50 (1949) 318-328.

\bibitem{bk}
R.H. Bruck and E. Kleinfeld, The structure of alternative division rings, Proc. Amer. Math. Soc. 2 (1951) 878-890.

\bibitem{hls}
J.T. Hartwig, D. Larsson, and S.D. Silvestrov, Deformations of Lie algebras using $\sigma$-derivations, J. Algebra 295 (2006) 314-361.

\bibitem{kleinfeld}
E. Kleinfeld, Right alternative rings, Proc. Amer. Math. Soc. 4 (1953) 939-944.

\bibitem{mak}
A. Makhlouf, Hom-alternative algebras and Hom-Jordan algebras, Int. Elect. J. Alg. 8 (2010) 177-190.

\bibitem{ms}
A. Makhlouf and S. Silvestrov, Hom-algebra structures, J. Gen. Lie Theory Appl. 2 (2008) 51-64.

\bibitem{ms2}
A. Makhlouf and S. Silvestrov, Hom-algebras and Hom-coalgebras, J. Algebra Appl. 9 (2010) 1-37.

\bibitem{maltsev}
A.I. Mal'tsev, Analytic loops, Mat. Sb. 36 (1955) 569-576.

\bibitem{schafer}
R.D. Schafer, An introduction to nonassociative algebras, Dover Pub., New York, 1995.

\bibitem{yau0}
D. Yau, Enveloping algebras of Hom-Lie algebras, J. Gen. Lie Theory Appl. 2 (2008) 95-108.

\bibitem{yau1}
D. Yau, Hom-algebras and homology, J. Lie Theory 19 (2009) 409-421.

\bibitem{yau2}
D. Yau, The Hom-Yang-Baxter equation, Hom-Lie algebras, and quasi-triangular bialgebras, J. Phys. A 42 (2009) 165202 (12pp).

\bibitem{yau3}
D. Yau, Hom-bialgebras and comodule Hom-algebras, Int. Elect. J. Alg. 8 (2010) 45-64.

\bibitem{yau4}
D. Yau, Hom-Maltsev, Hom-alternative, and Hom-Jordan algebras, arXiv:1002.3944.

\bibitem{yau15}
D. Yau, Hom-power associative algebras, arXiv:1007.4118.
\end{thebibliography}
\end{document}